\documentclass{amsart}

\usepackage{amssymb, amsmath, latexsym, a4}
\usepackage[latin1]{inputenc}
\usepackage[all]{xy}

\usepackage{tikz}
\usepackage{pinlabel}
\usepackage{amsfonts}
\usepackage{amscd}
\usepackage{txfonts}
\usepackage{geometry}

\usetikzlibrary{matrix,calc,arrows}
\newcommand*\Z{\mathbb{Z}}

\newcommand{\athir}[2]{\displaystyle \prod_{C_{#2}\subset G} H^{#1}( C_{#2},M)}

\newtheorem{De}{Definition}[section]
\newtheorem{Th}[De]{Theorem}
\newtheorem{Pro}[De]{Proposition}
\newtheorem{Le}[De]{Lemma}
\newtheorem{Co}[De]{Corollary}
\newtheorem{Rem}[De]{Remark}

\numberwithin{equation}{subsection}

\def\xto#1{\xrightarrow[]{#1}}

\newcommand{\w}{\mbox{\tiny $\wedge$ } }
\def\t{\otimes }

\def\la{\lambda }

\def\1{^{-1}}

\def \hom{\mathop{\sf Hom}\nolimits}
\def \ext{\mathop{\sf Ext}\nolimits}

\begin{document}

\title{Symmetric cohomology of groups}

\author[M. Pirashvili]{Mariam  Pirashvili}
\address{
Department of Mathematics\\
University of Southampton\\
University Road\\
Southampton\\
SO17 1BJ} \email{mp2m15@soton.ac.uk}
\thanks{The research was supported by the EPSRC grant EP/N014189/1 Joining the dots: from data to insight.}

\maketitle

\begin{abstract}  We investigate the relationship between the symmetric, exterior and classical cohomologies of groups. The first two theories were introduced respectively by Staic and  Zarelua. We show in particular, that there is a map from exterior cohomology to symmetric cohomology which is a split monomorphism in general and an isomorphism in many cases, but not always. We introduce two spectral sequences which help to explain the realtionship between these cohomology groups. As a sample application we obtain that symmetric and classical cohomologies are isomorphic for torsion free groups.  \end{abstract}

% 
%{\bf Keywords:}.
%

\section{Introduction}
Let $G$ be a group and $M$ be a $G$-module. In order to better understand 3-algebras arising in lattice field theory \cite{triangle}, Staic defined a variant of group cohomology, which he denoted by $HS^*(G,M)$ and called \emph{symmetric cohomology of groups} \cite{staic}. Some aspects of this theory were later extended by Singh \cite{singh} and Todea \cite{todea}. There is an obvious natural transformation from the symmetric cohomology to the classical Eilenberg-MacLane cohomology
$$\alpha^n:HS^n(G,M)\to H^n(G,M), \  \ n\geq 0.$$
According to \cite{staic},\cite{staic_h2},  $\alpha^n$ is an isomorphism if $n=0,1$ and is a monomorphism for $n=2$. By Corollary 2.3 in \cite{staic_h2} we know that $\alpha^2$ is an isomorphism if $G$ has no elements of order two. 
%Starting with a manifold $M$, such that $\pi_1(M)$ has no element of order two and three, Staic constructed an element in $HS^3(\pi_1(M),\pi_2(M))$ \cite[Corollary 4.4]{staic}, which maps via $\alpha^3$  to the classical first nontrivial $k$-invariant \cite[Remark 4.4]{staic}. 

Ten years prior to this, Zarelua had also defined a version of group cohomology, denoted by $H_\la^*(G,M)$ and called \emph{exterior cohomology of groups} \cite{zarelua}. It also comes together with a natural transformation 
$$\beta^n:H_\la^n(G,M)\to H^n(G,M),$$
with similar properties. The exterior cohomology  has the following striking property: If $G$ is a finite group of order $d$, then $H^i_\lambda(G,M)=0$ for all $i\geq d.$

The aim of this work is to obtain more information about homomorphisms $\alpha^* $ and $\beta^*$. We construct a natural transformation $\gamma^n:H_\la^n(G,M)\to HS^n(G,M)$ such that the following diagram commutes:
$$\xymatrix{H_\la^n(G,M)\ar[rr]^{\gamma^n}\ar[dr]_{\beta^n}&& HS^n(G,M)\ar[dl]^{\alpha^n}\\
&H^n(G,M).}$$
Our results in Section \ref{comp} show that the homomorphism $\gamma^n: H_\la^n(G,M)\to HS^n(G,M)$ is a split monomorphism in general, and an isomorphism in certain cases, namely if $0\leq n\leq 4,$ or $M$ has no elements of order two.
In general, $\gamma^5$ is not an isomorphism.

Our next results are related to the homomorphism $\beta^n:H_\la^n(G,M)\to H^n(G,A)$. We construct a spectral sequence for which $\beta^n$ are edge homomorphisms, $n \geq 0$. As any first quadrant spectral sequence, it gives a 5-term exact sequence (see for example \cite[Exercise 5.1.3]{weibel}) which has the following form:
$$0\to H_{\lambda}^2(G,M)\xto{\beta^2} H^2(G,M)\to \prod_{C_2\subset G}H^2(C_2,M)\to H^3_{\lambda}(G,M)\xto{\beta^3} H^3(G,M).$$
Here the product is taken over all subgroups of order two. The exactness at $H^2(G,M)$ is an answer to Problem 25 by Singh in \cite{problems}. At the very end of Section $4$ in \cite{staic}, Staic wondered about the injetivity of the map $\alpha^3$ under the assumption that $G$ has no elements of order $2$. A trivial consequence of our spectral sequence says that, if $G$ has no elements of order two, then one has an exact sequence:
$$0\to H_{\lambda}^3(G,M)\xto{\beta^3} H^3(G,M)\to \prod_{C_3\subset G}H^3(C_3,M)\to H^4_{\lambda}(G,M)\xto{\beta^4} $$ 
$$\xto{\beta^4}  H^4(G,M)\to  \prod_{C_3\subset G}H^4(C_3,M)\to H^5_{\lambda}(G,M)\xto{\beta^5} H^5(G,M).$$
In particular, if $G$ has no elements of order two and three, then $H^i_\lambda(G,M)=H^i(G,M)$, for $i=0,1,2,3,4$. 
%From this fact and \cite[Remark 4.5]{staic} it follows that the invariant of a manifold defined by Staic essentially coincides with the classical first nontrivial $k$-invariant.  
 
Among other consequences of our spectral sequence, we mention the following: if $G$ is a torsion free group, then $\beta^n:H^n_\lambda(G,M)\to H^n(G,M)$ is an isomorphism for all $n\geq 0$.
 
The paper is organised as follows: In Section 2 we recall the definitions of the symmetric and exterior cohomologies. In the next section we construct the transformation $\gamma^*$ and prove our first result, which shows that $\gamma^n$ is quite often an isomorphism, but not always. In the final section we construct a spectral sequence and we prove our main result Theorem \ref{e2}.
           
\section{Preliminaries}

\subsection{Classical cohomology} Let $G$ be a group and $M$ be a $G$-module. One way to define the cohomology $H^*(G,M)$ is via cochains, as $H^*(C^*(G,M))$. The group of $i$-cochains of $G$ with coefficients in $M$ is the set of functions from $G^i$ to $M$:
$$C^i(G,M)=\left\{\phi:G^i\to M\right\}.$$
The $i^{th}$ differential $\partial^i:C^i(G,M)\to C^{i+1}(G,M)$ is the map
\begin{align*}
\partial^i(\phi)(g_0,g_1,\cdots ,g_i)&=g_0\cdot \phi(g_1,\cdots ,g_i)\\
&+\sum_{j=1}^i(-1)^j\phi(g_0,\cdots ,g_{j-2},g_{j-1}g_j,g_{j+1},\cdots , g_i)\\
&+(-1)^{i+1}\phi(g_0,\cdots ,g_{i-1}).
\end{align*}
Given a chain complex such as this one, one can define its normalised subcomplex. In each dimension $n$, define $NC^n(G,M)$ to be the group of $n$-cochains which satisfy the normalisation condition
$$\phi(g_0,\cdots,g_{i-1},1,g_{i+1},\cdots,g_n)=0, \quad i=0,\cdots,n.$$
The canonical inclusion $\iota:NC^*(G,M)\to C^*(G,M)$ is a chain equivalence \cite{homology}.

Another way to define $H^*(G,M)$ is via projective resolutions, as $\emph{ H}^* (K^* (G,M))$. The standard projective resolution of $\Z$ by $G$-modules is the sequence of $G$-module homomorphisms \cite{aw}
$$\cdots \to \Z[G^{i+1}] \xrightarrow{\partial_{i-1}}\Z[G^i] \to\cdots \to\Z[G]\xrightarrow{\epsilon}\Z,$$
where
$$\partial_{i-1}(g_0,\cdots ,g_i)=\sum_{j=0}^i(-1)^j(g_0,\cdots ,g_{j-1},g_{j+1},\cdots , g_i),$$
and the mapping $\epsilon$ sends each generator $(g)$ to $1\in\Z$. 
An element of $$K^i(G,M)=\hom_G(\Z[G^{i+1}],M)$$ is then a function $f:G^{i+1}\to M$ such that
$$f(sg_0,sg_1,\cdots ,sg_i)=s\cdot f(g_0,g_1,\cdots,g_i).$$
The maps
$$K^i(G,M)\xrightarrow{\psi^i} C^i(G,M)$$
defined by
$$\psi^i(f)(g_1,\cdots , g_i)=f(1,g_1,g_1g_2, \cdots ,g_1g_2\cdots g_i)$$
induce an isomorphism of cochain complexes $K^*(G,M)\to C^*(G,M)$ \cite{aw}. Moreover, one has a commutative diagram
$$\xymatrix{K^*(G,M)\ar[r]^\psi& C^*(G,M)\\ NK^*(G,M)\ar[r]_\psi\ar[u]& NC^*(G,M)\ar[u]}$$
where the  horizontal maps are isomorphisms and the vertical maps are inclusions and homotopy equivalences. Here $NK^i(G,M)$ consists of such maps $f\in K^i(G,M)$ that
$$f(x_0,\cdots, x_i)=0, \ \ {\rm if} \ x_{j}=x_{j+1}, \ {\rm for} \ 0\leq j<n.$$ 
Thus
$$H^*(G,M)=H^*(NC^*(G,M))=H^*(C^*(G,M))=H^*(K^*(G,M))=H^*(NK^*(G,M)).$$
\subsection{Symmetric cohomology}
We now discuss a subcomplex of $C^*(G,M)$ introduced by Staic in \cite{staic}  and \cite{staic_h2}. It is based on an action of $\Sigma_{n+1}$ on $C^n(G,M)$ (for every $n$) compatible with the differential. In order to define this action, it is enough to define how the transpositions $\tau_i = (i,i + 1)$,
$1 \leq i \leq n$ act. For $\phi \in C^n(G,M)$ one defines:
$$(\tau_i \phi)(g_1, g_2, g_3, \cdots , g_n) = \begin{cases}  -g_1\phi(g_1^{-1}, g_1g_2, g_3,\cdots , g_n), \quad {\rm if}  \ i=1,\\ 
-\phi(g_1, \cdots , g_{i-2}, g_{i-1}g_i, g_i^{-1}, g_ig_{i+1},\cdots , g_n),&  1 < i < n,\\
-\phi(g_1, g_2, g_3,\cdots , g_{n-1}g_n, g_n^{-1}) , \quad  {\rm if}  \ i=n.
\end{cases}$$
Denote by  $CS^n(G,M)$ the subgroup of   the invariants of this action. That is, $CS^n(G,M)= C^n(G,M)^{\Sigma_{n+1}}$.
Staic proved  that $CS^*(G,M)$ is a subcomplex of $C^*(G,M)$ \cite{staic}, \cite{staic_h2}. 
\begin{De}The  homology of this subcomplex is called the symmetric cohomology of $G$ with coefficients in $M$ and is denoted by $HS^n(G,M)$.
\end{De}
\begin{Rem} There is a natural map  $\alpha^n:HS^n(G,M)\to H^n(G,M)$ induced by the inclusion  $CS^*(G,M)\hookrightarrow C^*(G,M)$.
\end{Rem}
\subsection{Exterior powers}
In order to define the chain complex introduced by Zarelua \cite{zarelua} we need to recall some facts about exterior powers.
\begin{De}
The exterior algebra $\Lambda^*(A)$ of an abelian group $A$ is a quotient algebra of the tensor algebra $T^*(A)$ with respect to the two-sided ideal generated by the elements of the form $a\otimes a \in T^2(A)=A\otimes A$.
\end{De}
A weaker version of this, denoted by $\tilde \Lambda^*(A)$, can be defined as the quotient algebra of the tensor algebra $T^*(A)$ with respect to the two-sided ideal generated by the elements of the form $a\otimes b +b\otimes a \in T^2(A)$.
Since $$a\t b +b\t a=(a+b)\t (a+b)-a\t a-b\t b,$$
it is clear that  one has the canonical quotient maps
$$\t^{n}(A)\twoheadrightarrow \tilde \Lambda^n(A)  \twoheadrightarrow  \Lambda^n(A).$$
Denote by $\Delta^n(A)$ the kernel of the projection $\tilde \Lambda^n(A)  \twoheadrightarrow  \Lambda^n(A)$. Thus we have a short exact sequence
$$0\to \Delta^n(A)\to \tilde \Lambda^n(A)  \twoheadrightarrow  \Lambda^n(A)\to 0.$$
Clearly $\Lambda^1(A)=A=\tilde\Lambda ^1(A)$. Hence
\begin{equation}\label{del1}
\Delta^1(A)=0.
\end{equation}

The images of $a_1\t\cdots \t a_n\in\t^{n}A$ in $\tilde \Lambda^n(A) $ and $\Lambda^n(A)$ are denoted  by
$a_1 \tilde \wedge \cdots \tilde \wedge a_n$ and $a_1  \wedge \cdots  \wedge a_n$ respectively. 
Recall that if $A=\Z[S]$ is a free abelian group with a set $S$ as basis, then $\t^{n}A$ is a free abelian group with basis elements $s_1\t\cdots\t s_n$, where $s_i\in S$. It is also well-known that $\Lambda^n(A)$ is a free abelian group with basis elements $s_1\wedge \cdots \wedge s_n$, where $s_1<\cdots <s_n$. Here $<$ is a total order on $S$.

In $\tilde \Lambda^n(A) $, $A=\Z[S]$, things are a  bit more complicated because of the relation $2a\tilde \wedge a=0$, which is a consequence of the relation $a\tilde \wedge b+b\tilde \wedge a=0$. It implies that $\Delta^n(A)$ is an  $\mathbb{F}_2$-vector space.
The epimorphism  $ \tilde \Lambda^n (\Z[S]) \to \Lambda^n(\Z[S])$ has a splitting given by  $s_1\wedge \cdots \wedge s_n\mapsto s_1\tilde \wedge \cdots \tilde \wedge s_n$. Here $s_1,\cdots,s_n$ are distinct elements in $S$. Thus 
\begin{equation}\label{tildasdas}
\tilde{\Lambda}^n(\Z[S])\cong \Lambda ^n(\Z[S])\oplus \Delta^n(\Z[S]), 
\end{equation}
Thus expressions of the form $s_1\tilde \wedge \cdots \tilde \wedge s_n$, where $s_1\leq \cdots \leq s_n$, are canonical generators of $\tilde\Lambda^n(\Z[S])$. Among these elements, ones with strict inequalities $s_1< \cdots < s_n$ form a basis of the summand corresponding to the free abelian group part, while the rest form a basis of the $\mathbb{F}_2$-vector space  $\Delta^n(\Z[S])$.

\subsection{Exterior cohomology of groups} We now discuss a subcomplex of $K^*(G,M)$, denoted by $K^*_{\lambda}(G,M)$, introduced by Zarelua in \cite{zarelua}. 

According to Lemma 3.1 in \cite{zarelua}, there is a differential $$\partial:\Lambda^{n+1}(\Z[G])\to\Lambda^n(\Z[G])$$ in the exterior algebra generated by $\Z[G]$ given by
$$\partial(g_0\wedge \cdots\wedge g_n)=\sum_{i=0}^n(-1)^{i+1}g_0\wedge \cdots\wedge \hat{g_i}\wedge\cdots\wedge g_n,$$
where, as usual, the hat  \ $\hat{}$ \ denotes a missing value. The group $G$ acts on this chain complex  by:$$g(g_1\wedge g_2\wedge\cdots\wedge g_n)=gg_1\wedge gg_2\wedge\cdots\wedge gg_n.$$
\begin{De}
The homology groups of the cochain complex (denoted by $K^*_\la(G,M)$)
$$\hom_G(\Lambda^1\Z[G],M)\xrightarrow{\partial}\hom_G(\Lambda^2\Z[G],M)\xrightarrow{\partial}\cdots\xrightarrow{\partial} \hom_G(\Lambda^n\Z[G],M)\xrightarrow{\partial}\cdots$$
are called the exterior cohomology groups of the group $G$ with coefficients in $M$ and are denoted by $H^n_{\lambda}(G,M)$. 
\end{De}

Therefore, $K^*_\la(G,M)$ is the subcomplex of $K^n (G,M)$ of all $G$-maps $f\in K^n(G,M)$ such that
$$f(g_0,\cdots,g_i,g_i,\cdots,g_n)=0, $$
and
$$f(g_0,\cdots,g_i,g_{i+1},\cdots,g_n)=- f(g_0,\cdots,g_{i+1},g_i,\cdots,g_n),$$
for all $0\leq i< n$.
%This condition implies skew-symmetry, leading us to conclude that $K^n_{\lambda}(G,M)$ is a subcomplex of $KS^n(G,M)$.
\begin{Rem} There is a natural transformation $\beta^n:H_\la^n(G,M)\to H^n(G,M)$ induced by the inclusion  $K_\la^*(G,M)\hookrightarrow K^*(G,M)$.
\end{Rem}

\begin{Rem}\label{26} Let $G$ be a finite group of order $d$. Since $\Z[G]$ is a free abelian group of rank $d$, we have  $\Lambda^i\Z[G]=0$, for $i>d$ and $H^n_\lambda(G,M)=0$ for $n\geq d$. On the other hand, as we will see later, the groups $HS^n(C_2,M)$ are nontrivial for infinitely many $n$.
\end{Rem}

\section{Comparison of symmetric and exterior cohomologies}\label{comp}
\subsection{Construction of the map $\gamma$}
We need two more complexes: $C_\la^*(G,M)$ and $KS^*(G,M)$. They are defined as follows.
\begin{De}
Let $KS^n (G,M)$ denote the subcomplex of $K^n (G,M)$ of all $G$-maps $f\in K^n(G,M)$ such that
\begin{equation}\label{KS}f(g_0,\cdots,g_i,g_{i+1},\cdots,g_n)=-f(g_0,\cdots,g_{i+1},g_i,\cdots,g_n)\end{equation}
for all $0\leq i< n$.
\end{De}
So we have the following subcomplexes:
$$K_\la^*(G,M) \hookrightarrow KS^*(G,M) \hookrightarrow K^*(G,M).$$
\begin{De}
Let $C^n_{\lambda}(G,M)$ be the complex defined by $$C^*_\la(G,M)= CS^n(G,M) \cap NC^*(G,N)$$
Thus $\phi\in  CS^n(G,M)$ belongs to $C^n_{\lambda}(G,M)$ if
$$\phi(x_1,\cdots, 1,\cdots,x_n)=0.$$
\end{De}
This subcomplex has already been considered by \cite{todea}, who showed that if $M$ has no elements of order $2$, then $C^n_{\lambda}(G,M) = CS^n(G,M)$ for all $n$. We will later prove the same fact in a different way.

We have the following subcomplexes:
$$C_\la^*(G,M)  \hookrightarrow  CS^*(G,M)  \hookrightarrow C^*(G,M).$$
In order to understand the relationship between all these complexes it is useful to rewrite them in terms of resolutions, which we constructed in Lemma \ref{3res} below.

Since $\Z[G^i]=\Z[G]^{\t i}$, the standard projective resolution can be rewritten as
$$\cdots \to \Z[G]^{\t i}\to \Z[G]^{\t i-1}\to\cdots \to \Z[G]^{\t 2} \to \Z[G].$$
%which we denote by $(T^*(\Z[G]), \partial)$. 
If one replaces the tensor algebra by either version of the exterior algebra, one still obtains a resolution, though in general no longer a projective one. This is the subject of the following lemma.
\begin{Le}\label{3res} One has a commutative diagram of resolutions of $\Z$:
$$\xymatrix{\cdots \ar[r]& \Z[G]^{\t i}\ar[r]\ar[d]& \Z[G]^{\t i-1}\ar[r]\ar[d]&\cdots\ar[r] &  \Z[G]^{\t 2} \ar[r]\ar[d]& \Z[G]\ar[d]^{Id}\\
\cdots \ar[r]& \tilde\Lambda^i\Z[G] \ar[r]\ar[d] & \tilde\Lambda^{i-1}\Z[G]\ar[r]\ar[d]&\cdots\ar[r] & \tilde\Lambda^2 \Z[G] \ar[r]\ar[d] & \Z[G]\ar[d]^{Id}\\
\cdots \ar[r]& \Lambda^i\Z[G] \ar[r]& \Lambda^{i-1}\Z[G]\ar[r]&\cdots\ar[r] & \Lambda^2 \Z[G] \ar[r]& \Z[G]}$$
%The chain complex 
%$$\cdots \to\Lambda^{i+1} \Z[G] \xrightarrow{\partial_{i-1}}\Lambda^i\Z[G] \to\cdots \to\Z[G]\xrightarrow{\epsilon}\Z$$
%is a resolution for $Z$ (though generally not projective).
\end{Le}

One denotes these resolutions by $(T^{*+1}(\Z[G]),\partial)$, $(\tilde \Lambda^{*+1}(\Z[G]), \partial)$ and $(\Lambda^{*+1}(\Z[G]),\partial)$ respectively.

\begin{proof} 
In this proof, take $\partial_{-1} = \epsilon:\Z[G]\to \Z$. We only present the proof for $\Lambda^*$, as the proof for $\tilde \Lambda^*$ is similar.
We construct a homomorphism $h:\Lambda^i\Z[G]\to\Lambda^{i+1}\Z[G]$ by the formula
$$h(g_0\wedge\cdots\wedge g_i)=1\wedge g_0\wedge\cdots\wedge g_i.$$
To show that this is a contracting homotopy, we need to check that $h\circ \partial +\partial\circ h = Id_{\Lambda^i\Z[G]}$.
Indeed, we have
\begin{align*}
\partial_i\circ h_i(g_0\wedge\cdots\wedge g_i)&=g_0\wedge\cdots\wedge g_i -\sum^i_{j=0}(-1)^j1\wedge g_0\wedge\cdots\wedge \hat{g_j}\wedge\cdots\wedge g_i\\
&=g_0\wedge\cdots\wedge g_i - h_{i-1}\circ\partial_{i-1}(g_0\wedge\cdots\wedge g_i).
\end{align*}
\end{proof}

\begin{Le}\label{34mai} The differential $\partial: \tilde\Lambda^{n+1}(\Z[G])\to \tilde\Lambda^n(\Z[G])$ sends $\Delta^{n+1}(\Z[G])$ to $\Delta^n(\Z[G])$. Moreover, it is compatible with the decompostion (\ref{tildasdas}). Hence 
$$(\tilde{\Lambda}^{*+1}(\Z[G]),\partial)\cong (\Delta^{*+1}(\Z[G]),\partial)\oplus (\Lambda^{*+1}(\Z[G]),\partial)$$
and $H_*(\Delta^{*+1}(\Z[G]),\partial)=0$.
\end{Le}
\begin{proof} By Lemma \ref{3res} the canonical projection $\tilde \Lambda^{*+1}(\Z[G]) \to \Lambda^{*+1}(\Z[G]))$ is a chain map, inducing an isomorphism in homology, hence  $\Delta^{*+1}(\Z[G])$ is a chain subcomplex with trivial homology. To finish the proof, it suffices to note that the map $g_1\wedge \cdots \wedge g_n\mapsto g_1\tilde \wedge \cdots \tilde \wedge g_n$, $g_1,\cdots,g_n\in G$, commutes with differentials and hence defines a splitting of chain complexes.

\end{proof} 

\begin{Le}\label{315} After  applying the functor $\hom_{\Z[G]}(-,M)$ to the resolutions in Lemma \ref{3res} one obtains the following diagram
$$\xymatrix{\hom_{\Z[G]}(\Lambda^*(\Z[G]),M) \ar[r] \ar[d]^{=} & \hom_{\Z[G]}(\tilde\Lambda^*(\Z[G]),M) \ar[r]\ar[d]^=&  \hom_{\Z[G]}(T^*(\Z[G]),M)  \ar[d]^{=}\\
K_\la^*(G,M)\ar[d]^\psi\ar[r]& KS^*(G,M)\ar[r]\ar[d]^\psi & K^*(G,M)\ar[d]^\psi\\
C_\la^*(G,M)\ar[r]& CS^*(G,M) \ar[r] & C^*(G,M)}$$
where all horizontal arrows are inclusions and vertical arrows are isomorphisms.
\end{Le}
\begin{proof} A key point is to show that restricting $\psi$ on $KS^*(G,M)$ yields an isomorphism between $KS^*(G,M)$ and $CS^*(G,M)$.
To this end, take  $\phi\in CS^n(G,M)$. Then
$$f(g_0,\cdots,g_n)=(\psi^n)^{-1}(\phi)(g_0,g_1,\cdots,g_n)=g_0\cdot\phi(g_0^{-1}g_1,g_1^{-1}g_2,\cdots,g_{n-1}^{-1}g_n).$$
The equation $\phi(g_1, g_2,g_3,\cdots , g_n) = -g_1\phi(g_1^{-1}, g_1g_2, g_3,\cdots , g_n)$ translates to
\begin{align*}
f(g_0,\cdots,g_n)&= g_0(\psi^n)^{-1}(\phi)(g_0^{-1}g_1,\cdots,g_{n-1}^{-1}g_n)\\
&=-g_0\cdot g_0^{-1}g_1\cdot (\psi^n)^{-1}(\phi)((g_0^{-1}g_1)^{-1},g_0^{-1}g_1g_1^{-1}g_2,\cdots,g_{n-1}^{-1}g_n)\\
&=-g_1\cdot (\psi^n)^{-1}(\phi)(g_1)^{-1}g_0,g_0^{-1}g_2,\cdots,g_{n-1}^{-1}g_n)\\
&= -f(g_1,g_0,\cdots,g_n).
\end{align*}
In a similar way, the other equations above give the condition \ref{KS} for $n>0$.
\end{proof}
If one passes to cohomology, one obtains the homomorphisms
$$\xymatrix{H^*_\la(G,M)\ar[r]\ar[dr]_\gamma & H^*( KS^*(G,M))\ar[d]^{\cong}\ar[r] &H^*(G,M)\\ & HS^*(G,M)\ar[ru]_\alpha &}$$
and the commutativity of the diagram in Lemma \ref{315} shows that $\beta=\alpha\gamma$. 
\begin{Pro}\label{zdast} The homomorphism $\gamma^n: H_\la^n(G,M)\to HS^n(G,M)$ is a split monomorphism. Moreover, it is an isomorphism provided  $M$ has no elements of order two.
\end{Pro}

\begin{proof}The first part follows from Lemma \ref{34mai}. Assume $M$ has no elements of order two. It suffices to show that $K_\la^*(G,M)=KS^*(G,M)$. Take an element $f\in KS^n(G,M)$. Then we have
 $$f(x_0,\cdots,x_i,x_{i+1},\cdots, x_n)=- f(x_0,\cdots,x_{i+1},x_{i},\cdots, x_n),$$
 for all $0\leq i<n$. If  $x_i=x_{i+1}$, one obtains $2f(x_0,\cdots,x_i,x_{i},\cdots, x_n)=0$ and hence $f(x_0,\cdots,x_i,x_{i},\cdots, x_n)=0$. This implies that $f\in K^n_\la(G,M)$ and the proof is finished.
\end{proof}

\subsection{$\delta$-cohomology} In order to state the realtionship between the exterior and symmetric cohomology we need to introduce new groups.
\begin{De} For a group $G$ and a $G$-module $M$ one defines the $\delta$-homology $H^*_\delta(G,M)$ by
$$H^*_\delta(G,M)=H^*(\hom_{\Z[G]}(\Delta^{*+1}(\Z[G]),M)).$$
\end{De}
Since $\Delta^n(\Z[G])$ is an $\mathbb{F}_2$-vector space, it follows that the groups $H^n_\delta(G,M)$ are also  $\mathbb{F}_2$-vector spaces, $n\geq 0$. The importance of these groups comes from the fact that
\begin{equation}\label{hs=hl+hd}
HS^n(G,M)\cong H_\lambda^n(G,M)\oplus H^n_\delta(G,M)
\end{equation}
which is a trivial consequence of Lemma \ref{34mai}. It follows from Proposition \ref{zdast} that if $M$ has no elements of order two, then $H_\delta^*(G,M)=0$.

\subsection{Preliminaries on spectral sequences}\label{hhs}

To state our main result of this section, let us  recall the construction of the hypercohomology spectral sequences.  These spectral sequences will also play a prominent role in the next section.   

Let $G$ be a group and $M$ be a left $G$-module.  For any chain complex of left $G$-modules $C_*=(C_0\leftarrow C_1\leftarrow\cdots)$ one defines $\ext_{\Z[G]}^*(C_*,N)$ to be the homology of the total complex of the bicomplex $\hom_{\Z[G]}(C_*,I^*)$, where $I^*$ is an injective resolution of $M$.

There exist two spectral sequences.
%called respectively the first and second hypercohomology spectral sequences. 
Both of them abut to the group $\ext_{\Z[G]}^*(C_*,M)$. They are:
$$\mathbf{I}^{pq}_1 = \ext_{\Z[G]}^q(C_p,N)\quad \Longrightarrow \quad \ext_{\Z[G]}^{p+q}(C_*,M),$$
$$\mathbf{II}^{pq}_2 = \ext_{\Z[G]}^p(H_q(C_*),M)\quad \Longrightarrow \quad \ext_{\Z[G]}^{p+q}(C_*,M).$$
We also need the following easy lemma on spectral sequences
\begin{Le}\label{abut0} Assume a spectral sequence %$E^{pq}_2\quad \Longrightarrow 0$
abuts to zero and $E_2^{pq}=0$ if $q<0$ or $p<k$, where $k$ is a fixed integer. Then $$E_2^{k\,0}=0=E^{k+1\,0}_2.$$ 
%and the differential $d^2$ induces an isomorphism
%$$E_2^{k\,1}\xto{\sim} E_{2}^{k+2\,0}.$$
\end{Le}
\subsection{Vanishing of $\delta$-cohomology in low dimensions}
Now we can state the main result of this section:
\begin{Th}  Let $G$ be a group and $M$ be a $G$-module. Then
$$H^i_\delta(G,M)=0, \quad {\rm for} \quad 0\leq i\leq 4.$$ 
Hence $\gamma^i:H^i_\lambda(G,M)\to HS^i(G,M)$ is an isomorphism for $i=0,1,2,3,4$.
\end{Th}

\begin{proof} In the hypercohomology spectral sequence we take
%Take $$R=\mathbf{F}_2[G],\quad N=\, _2M \quad {\rm and} \quad  
$C_*=(\Delta^{*+1}(\Z[G]),\partial).$ Since $H_*(C_*)=0$, the spectral sequence $\mathbf{II}$ gives $\ext_{\Z[G]}^*(C_*,M)=0$. Thus, the spectral sequence $\mathbf{I}$ has the form
$$E^{pq}_1=\ext_{\mathbb{Z}[G]}^q(\Delta^{p+1}(\Z[G]),M)\quad\Longrightarrow 0.$$
Since $E^{p0}_1=\hom_{\Z[G]}(\Delta^{p+1}(\Z[G]),M),$
we see that $$E_2^{*0}=H^*_\delta(G,M).$$ According to  (\ref{del1}) we have $E^{pq}_1=0$ for $p< 1$. It follows from Lemma \ref{abut0} that $E^{i\,0}_2=H^i_\delta(G,M)=0$ for $i\leq 2$. Thus by  the same Lemma it suffices to show that
 $E^{1,q}_2=0=E^{2,q}_2$ if $q>0$.
 
One checks that the following diagram of $G$-modules commutes:
$$\xymatrix{\Delta^2(\Z[G])&\Delta^3(\Z[G])\ar[l]_{\partial^1}& \Delta^4(\Z[G])\ar[l]_{\partial^2}\\
\mathbb{F}_2[G] \ar[u]_{\psi_1}& \mathbb{F}_2[G\times G] \ar[l]_{\delta_1}\ar[u]_{\psi_2}& \mathbb{F}_2[G\times G] \ar[l]_{\delta_2}\ar[u]_{\psi_3}
}$$ 
where
$$\psi_1(g)=g\tilde\wedge g, \quad \psi_2(g,h)=g\tilde\wedge g\tilde\wedge h, \quad \psi_3(g,h)=g\tilde\wedge g\tilde\wedge g\tilde\wedge h,$$
$$\delta_1(g,h)=g, \quad \delta_2(g,h)=(g,h)-(g,g).$$
Since  the set  of elements $\{s\tilde \wedge s|s\in G\}$, (resp.  $\{s\tilde\wedge s\tilde \wedge t| s,t\in G\}$) forms an $\mathbb{F}_2$-basis of $\Delta^2(\Z[G])$ (resp. $\Delta^3(\Z[G])$), the $G$-homomorphism $\psi_1$ (resp. $\psi_2$) is an isomorphism. In general, the $G$-homomorphism $\psi_3$ is not an isomorphism, but only a split monomorphism. 
Hence the projective resolutions
$$0\to \Z[G]\xto{2}\Z[G]\to \mathbb{F}_2[G]\to 0 \quad {\rm and} \quad 0\to \Z[G\times G]\xto{2}\Z[G\times G]\to \mathbb{F}_2[G\times G]\to 0$$
can be used to compute  $\ext_{\mathbb{Z}[G]}^q(\Delta^{2}(\Z[G]),M)$ and $\ext_{\mathbb{Z}[G]}^i(\Delta^{3}(\Z[G]),M)$. In both cases $$Ext_{\Z[G]}^i(\Delta^2(\Z[G]),M)=0=Ext_{\Z[G]}^i(\Delta^3(\Z[G]),M) \quad {\rm if} \quad i>1.$$
Hence $E^{1,q}_1=0=E^{2,q}_1$ if $q>1$. The first projective resolution gives 
$$E^{11}_1=Ext_{\Z[G]}^1(\Delta^2(\Z[G]),M)=N,$$
where $N= M/2M.$ Since $\Z[G\times G]=\oplus _{g\in G}\Z[G]$ as a $G$-module, the second projective resolution gives
$$E^{21}_1=Ext_{\Z[G]}^1(\Delta^3(\Z[G]),M)=\\ Maps(G,N).$$
Moreover, it also shows that the group $ Maps(G,N)$ is a direct summand of $Ext_{\Z[G]}^1(\Delta^4(\Z[G]),M)$. It follows that there is an isomorphism of chain complexes
$$\xymatrix{E^{01}_1\ar[r]^{\partial^0}&E^{11}_1\ar[r] ^{\partial^1}&E^{21}_1\ar[r]^{\partial^2}& E^{31}_1\\
0\ar[r]^{\delta^0}\ar[u]&N \ar[r] ^{\delta^1} \ar[u]_{\psi^*_1}& Maps(G,N)\ar[r]^{\delta^2}\ar[u]_{\psi^*_2}& X\oplus Maps(G,N)\ar[u]_{\psi^*_3}
}$$ 
for some $X$, where $(\delta^1(n))(g)=n$,  $\delta^2=\begin{pmatrix}x\\ \delta'\end{pmatrix}$ for some $x$ and 
$(\delta'(\tau))(g)=\tau(g)-\tau(1)$. Since $\delta^1$ is a monomorphism, it follows that $E^{1,1}_2=0$. And as $Ker(\delta')=Im(\delta^1)$, we obtain that $E^{2,1}_2=0$ and the proof is finished.
\end{proof}
Now we give an example which shows that $\gamma^n$, $n\geq 5$ is not an isomorphism in general.
\subsection{The symmetric and exterior cohomologies of $C_2$} Let $G=C_2=\{1,t\}$, $t^2=1$ be the cyclic group of order two. In this section we compute both symmetric and exterior cohomologies of $C_2$. The computation of the exterior cohomology is extremely easy. In fact, for $G=C_2$, the resolution $(\Lambda ^*(\Z[G]),\partial)$  has the following form:
$$\xymatrix{ \cdots \ar[r]&0\ar[r] &\Lambda^3(\Z[C_2])\ar[r]^\partial \ar[d]^\cong\ar[r]&\Lambda^2(\Z[C_2])\ar[r]^\partial \ar[d]^\cong &  \Z[G]\\
&&0& \Z[C_2]/(1+t)&}$$
where $\partial =(1-t)$.
So,
$$H_\la^n(C_2,M)=\begin{cases} H^n(C_2,M), \quad {\rm if} \ n=0,1, \\ 0, \quad {\rm else.}\end{cases}$$ 
%H^n(M\xto{d} \{m\in M| tm+m=0\}\to 0\to \cdots)
For the symmetric cohomology one has the following result:
\begin{Le} For $G=C_2$ and $M=\mathbb{F}_2$ with trivial action of $G$ on $M$, one  has
\begin{equation*}
HS^i(C_2,\mathbb{F}_2) = \begin{cases}
\mathbb{F}_2, & \quad {\rm if} \quad i=0, \quad {\rm or}\quad  i\equiv 1 \quad mod\quad 4,\\
0, & \quad {\rm else}.
\end{cases}
\end{equation*}
\end{Le}

Thus, in general, $H_{\lambda}^*(G,M)\neq HS^*(G,M).$
%For example, take $G=C_2$ and $M=\mathbb{F}_2$ with trivial action of $G$ on $M$. 
%Then we have
%\begin{equation*}
%HS^i(C_2,\mathbb{F}_2) = \begin{cases}
%\mathbb{F}_2, & \quad if \quad i=0, \quad or\quad  i\equiv 1 \quad mod\quad 4,\\
%0, & \quad else.
%\end{cases}
%\end{equation*}Let us demonstrate this
\begin{proof} Consider the resolution
$$\cdots \to\tilde\Lambda^3 \Z[C_2] \xrightarrow{\partial_1}\tilde\Lambda^2\Z[C_2] \xto{\partial_0}\Z[C_2].$$
Fix $n> 0$ and in $\tilde\Lambda^n\Z[C_2]$ consider the elements
$$\alpha^n_i=\underbrace{1\tilde\w\cdots\tilde\w 1}_{n-i}\tilde\w\underbrace{t\tilde\w\cdots\tilde\w t}_{i}, \quad 0\leq i< \frac{n}{2},$$
$$\beta^n=\underbrace{1\tilde\w\cdots\tilde\w 1}_{m}\tilde\w\underbrace{t\tilde\w\cdots\tilde\w t}_{m}, \quad n=2m.$$

Then $\alpha^n_i$ and $\beta^n$ generate $\tilde\Lambda^n \Z[C_2]$ as a $C_2$-module. More accurately, $\Z[C_2]$ is a free $\Z[C_2]$-module with the generator 
$\alpha^1_0$. As  a $\Z[C_2]$-module,
$$\tilde\Lambda^2\Z[C_2]=\mathbb{F}_2[C_2]\bigoplus \Z[C_2]/(t+1),$$
with $\alpha^2_0$ generating $\mathbb{F}_2[C_2]$ and $\beta^2$ generating $\Z[C_2]/(t+1)$. For odd $n$, $n=2m+1\geq 3$,
$$\tilde\Lambda^n\Z[C_2]=\mathbb{F}_2[C_2]\bigoplus\cdots\bigoplus \mathbb{F}_2[C_2],$$
with $\alpha^n_0,\cdots,\alpha^n_m$ generating each of the summands. Similarly to $n=2,$ for larger $n=2m\geq 4$ we have
$$\tilde\Lambda^n\Z[C_2]=\mathbb{F}_2[C_2]\bigoplus\cdots\bigoplus \mathbb{F}_2[C_2]\bigoplus \mathbb{F}_2[C_2]/(t-1),$$
where the $\alpha^n_0,\cdots,\alpha^n_m$ generate the $\mathbb{F}_2[C_2]$ summands and $\beta^n$ generates $\mathbb{F}_2[C_2]/(t-1)$.

Beginning from $\partial_1$, the coboundary maps are given by the matrices
\begin{equation*}
(\partial_{4k+1})_{ij} = \begin{cases}
1, & \quad if \quad i \quad is \quad odd \quad and\quad  j=i \quad or \quad j=i+1\\
0, & \quad else,
\end{cases}
\end{equation*}
where $1\leq i\leq 2k+2$, $1\leq j\leq 2k+2$,

\begin{equation*}
(\partial_{4k+2})_{ij} = \begin{cases}
1, & \quad if \quad j \quad is \quad even \quad and\quad  i=j \quad or \quad i=j-1\\
0, & \quad else,
\end{cases}
\end{equation*}
where $1\leq i\leq 2k+2$, $1\leq j\leq 2k+3$,

\begin{equation*}
(\partial_{4k+3})_{ij} = \begin{cases}
1, & \quad if \quad i \quad is \quad odd \quad and\quad  j=i\\
1,  & \quad if \quad i \quad is \quad odd \quad and\quad  j=i+1 \quad and \quad i<2k+2\\
0, & \quad else,
\end{cases}
\end{equation*}
where $1\leq i\leq 2k+3$, $1\leq j\leq 2k+3$,
\begin{equation*}
(\partial_{4k})_{ij} = \begin{cases}
1, & \quad if \quad j \quad is \quad even \quad and\quad  i=j \quad or \quad i=j-1 \quad and \quad j<2k+2\\
t-1, & \quad if \quad j=2k+2 \quad and \quad i=2k+1\\
0, & \quad else,
\end{cases}
\end{equation*}
where $1\leq i\leq 2k+1$, $1\leq j\leq 2k+2$.
Based on this the result easily follows.
\end{proof}

\section{Relationship between exterior and classical cohomology}

We start this section with the following easy and probably well-known fact. It will be used in the proof of Theorem \ref{e2} below.
\begin{Le}\label{divides}
Let $g\in G$ and $\omega=x_1\wedge \cdots\wedge x_n\in \Lambda^n\Z[G]$, where $x_1,\cdots,x_n$ are distinct elements in $G$. If $g\omega=\pm \omega$, then the order of $g$ divides $n$.
\end{Le}
\begin{proof} If one forgets the sign, it follows from the assumption that the multiplication by $g$ permutes the $n$ elements $x_1,\cdots , x_n$, meaning the cyclic group generated by $g$ acts on the set $\{x_1,\cdots,x_n\}$. The action is free, because it is given by the  multiplication in $G$. Hence all orbits will have the same length equal to the order of $g$, dividing $n$.  
\end{proof}

Now we can state our main result.

\begin{Th}\label{e2} For any group $G$ and any $G$-module $M$, there is a first quadrant  spectral sequence
  $$E^{pq}_1\Longrightarrow H^{p+q}(G,M)$$
  with properties  
  \begin{itemize} 
  \item[(i)] $E^{p,0}_2=H^p_\lambda(G,M)$ and the edge homomorphism $E^{p0}_2\to H^p(G,M)$ is precisely $\beta^p$, $p\geq 0$.
  \item[(ii)] If $q>0$, then  $E^{0q}_1=0$. 
  \item[(iii)] If $q>0$,  $p>0$ and the equation $x^{p+1}=1$ has only trivial solution in $G$, then $E^{pq}_1=0$.
 % \item[(iii)] If $m>0$, then
  %$$E^{1m}_2=\prod_{C_2\subset G}H^{m+1}(C_2,M).$$ Here product is taken over all subgroups of order two and for each such subgroup, the corresponding action of $C_2$ on $M$ is induced by the inclusion.  
 \item[(iv)] If $\ell$ is a prime number and $q>0$, then 
  $$E^{\ell-1 \, q}_1=\begin{cases} \ \prod_{C_\ell\subset G}H^{q+1}(C_\ell,M),  \ {\rm if} \ \ell=2,\\
 \ \prod_{C_\ell\subset G}H^{q}(C_\ell,M), \ {\rm if} \  \ell>2.
\end{cases}$$
Here the product is taken over all subgroups of order $\ell$ and for each such subgroup, the corresponding action of $C_\ell$ on $M$ is induced by the inclusion. 
  \end{itemize}
  \end{Th}
  \begin{Rem} If $p+1 $ is not prime, then $E^{pq}_1$, $q>0, p>0$ can be described as a product  (usually of several copies) of the group cohomology of subgroups of order $k$, where $k|p+1$, but the exact formula is too complex to state here. From this it is easy to deduce that $E^{pq}_1=E^{pq}_2$ for all $q>0$ (compare with the proof of the part i) of Corollary \ref{orisami}).
  \end{Rem}
\begin{proof} In the hypercohomology spectral sequence discussed in Section \ref{hhs},  we take $R=\Z[G]$, $N=M$ and $C^*=(\Lambda^{*+1}(\Z[G]),\partial)$, which  we denote simply by $\Lambda^{*+1}$.  This gives the spectral sequences
$$\mathbf{I}^{pq}_1 = \ext_G^q(\Lambda^{p+1},M)\quad \Longrightarrow \quad \ext_G^{p+q}(\Lambda^{*+1},M)$$
$$\mathbf{II}^{pq}_2 = \ext_G^p(H_q(\Lambda^{*+1}),M)\quad \Longrightarrow \quad \ext_G^{p+q}(\Lambda^{*+1},M).$$
Let us first consider the second spectral sequence. As $\Lambda^{*+1}$ is a resolution of $\Z$, we have
\begin{equation*}
H_q(\Lambda^{*+1})= \begin{cases}
\Z, & \quad {\rm for} \quad q=0,\\
0, & \quad {\rm else}.
\end{cases}
\end{equation*}
Therefore, the second spectral sequence degenerates to the isomorphism
$$\ext_G^p(\Lambda^{*+1},M)=\ext_G^p(H_0(\Lambda^{*+1}),M)=\ext_G^p(\Z,M)=H^p(G,M).$$
%$$\ext_G^p(H_0(\Lambda^*),M)=\ext_G^p(\Z,M)=H^p(G,M)=\ext_G^p(\Lambda^*,M).$$
Substituting this value into the first spectral sequence, we obtain the spectral sequence
$$E^{pq}_1=\ext^{q}_{\Z[G]}(\Lambda ^{p+1}(\Z[G]), M)\Longrightarrow H^{p+q}(G,M).$$
Since the differential $d_1$ in the first page of the spectral sequence is induced by the boundary map in the resolution $\Lambda^{*+1}(\Z[G])\to \Z$, it follows that for $q=0$, the chain complex $(E^{p0}_1,d^1)$ coincides with the  Zarelua chain complex and the statement (i) follows.

If $p=0$, then $E^{pq}_1=Ext^q_{\Z[G]}(\Z[G],M)$  vanishes for $q>0$. Hence $E^{0q}_1=0$ for $q>0$, and  the property (ii) holds.

Next, the $G$-module $\Lambda ^{q+1}(\Z[G])$ is free as an abelian group with a basis of the form $x_1\wedge \cdots\wedge x_p$, where $x_1<\cdots<x_p$. Here $\leq$ is any total order on $G$. If one ignores the sign, we see that $G$ acts on the basis. Thus $\Lambda ^{p+1}(\Z[G])$ decomposes as a direct sum of $G$-submodules corresponding to these orbits. In particular, summands corresponding to  free orbits are free $G$-modules.  Now, if the assertion of (iii) holds, all orbits are free thanks to Lemma \ref{divides} and hence the $Ext$-group vanishes and $E^{pq}_1=0$ for $q>0$. Thus  the property (iii) is proved.

If $\ell$ is prime and $C_\ell=\{1,g,\cdots, g^{p-1}\}$ is a cyclic subgroup of $G$, then for the basis element $\omega=1\wedge g\wedge \cdots \wedge g^{\ell-1}$ one has $g\omega =\omega$ for odd $\ell$, and $g\omega =-\omega$ for $\ell=2$. Thus $\omega$ determines a non free summand of $\Lambda^{p}(\Z[G])$. This summand is isomorphic to $\Z[G]/(g-1)$ for odd $\ell$ and $\Z[G]/(g+1)$ for $\ell=2$. This summand has an obvious projective resolution
$$0\leftarrow \Z[G]/_{(g-1)}\leftarrow  \Z[G]\xleftarrow{g-1}\Z[G]\xleftarrow{1+g+\cdots g^{\ell-1}}\cdots $$
if $\ell$ is odd and 
$$0\leftarrow \Z[G]/_{(g+1)}\leftarrow  \Z[G]\xleftarrow{g+1}\Z[G]\xleftarrow{g-1}\cdots $$
if $\ell=2$. From this it follows that this summand of $\Lambda^{\ell}(\Z[G])$ contributes the factor $H^{i}(C_\ell,M)$ (resp. $H^{i+1}(C_\ell,M)$) in $Ext^{m}_{\Z[G]}(\Lambda ^{p+1}(\Z[G]), M)$ for odd $\ell$ (resp. $\ell=2$). By Lemma \ref{divides} all non-free summands of $\Lambda^{\ell}(\Z[G])$ arise in this way and hence $E^{\ell-1\,m}_1$ has the form described in
 (iv).  
 %It remains to show that $E^{\ell-1\,m}_1=E^{\ell-1\,m}_2$ for $m>0$. 
\end{proof}

Thus the first plane/page of the  spectral sequence is:

\begin{tikzpicture}
\matrix (m) [matrix of math nodes,
             nodes in empty cells,
             nodes={minimum width=10ex,
                    minimum height=10ex,
                    outer sep=-5pt},
             column sep=1ex, row sep=1ex,
             text centered,anchor=center]{
    q\strut   &  0 \strut  &  \athir{q+1}{2}   & \athir{q}{3}    & \cdots & Ext^q(\Lambda^{p+1}\Z [G],M)& \cdots  \\
    \vdots&  \vdots   &  \vdots         & \vdots          & \ddots & \vdots       & \cdots \\
	2      &  0   &  \athir{3}{2}   & \athir{2}{3}    & \cdots & Ext^2(\Lambda^{p+1}\Z [G],M) & \cdots \\
	1      &  0   &  \athir{2}{2}  & \athir{1}{3}   & \cdots & Ext^1(\Lambda^{p+1}\Z [G],M) & \cdots \\
	0      &  H_\la^0(G,M) & H_\la^1(G,M) &  H_\la^2(G,M) & \cdots  &  H_\la^p(G,M) & \cdots \\
  \quad\strut &   0 &  1  &  2  &  \cdots &  p  &  \strut \\};
    %\draw[-stealth] (m-1-3) -- (m-1-2);
\draw[thick] (m-1-1.north east) -- (m-6-1.east) ;
\draw[thick] (m-6-1.north) -- (m-6-7.north east) ;
\end{tikzpicture}

As an immediate consequence of Theorem \ref{e2} one obtains the following corollary.
\begin{Co}\label{orisami}
\begin{itemize}
\item[(i)] For any group $G$ and any $G$-module $M$,  the homomorphism $\beta^i:H^i_\lambda(G,M)\to H^i(G,M)$ is an isomorphism for $i=0$ and $i=1$, while $\beta^2$ and $\beta^3$ can be fit in an exact sequence:
$$0\to H_{\lambda}^2(G,M)\xto{\beta^2} H^2(G,M)\to \prod_{C_2\subset G}H^2(C_2,M)\to H^3_{\lambda}(G,M)\xto{\beta^3} H^3(G,M).$$

\item[(ii)] If $G$ has no elements of order two, then for any $G$-module $M$, the homomorphism $\beta^2$ is an isomorphism, while $\beta^3, \beta^4$ and $\beta^5$ can be fit in an exact sequence:
$$0\to H_{\lambda}^3(G,M)\xto{\beta^3} H^3(G,M)\to \prod_{C_3\subset G}H^3(C_3,M)\to H^4_{\lambda}(G,M)\xto{\beta^4} $$
$$\xto{\beta^4} H^4(G,M)\to  \prod_{C_3\subset G}H^4(C_3,M)\to H^5_{\lambda}(G,M)\xto{\beta^5} H^5(G,M).$$

\item[(iii)] If all nontrivial elements of $G$ are of infinite order, then $\beta^i:H^i_\lambda(G,M)\to H^i(G,M)$ is an isomorphism for all $i\geq 0$.
\end{itemize}
\end{Co}

\begin{proof}
\begin{itemize}
\item[(i)] We first show that if $q>0$, the differential $E^{1q}_1\to E^{2q}_1$ vanishes. In fact, by part (iv) of Theorem \ref{e2} the group $E^{1q}_1$ is annihilated by the multiplication by 2, while the group $E^{1q}_1$ is annihilated by the multiplication by 3 and hence the corresponding map is zero. This fact implies that
$E^{1q}_2=E^{1q}_1$ for all $q>0$.  The rest is  a consequence of the 5-term exact sequence, which we have in any first quadrant spectral sequence.  

\item[(ii)] Assume $q>0$. By part (iii) of Theorem \ref{e2} and the fact that $G$ does not contain an element of order two, we have $E^{pq}_0=0$, if $q>0$ and $p+1$ is a power of two. It follows that $E^{pq}_2=E^{pq}_1$, for $p=2$ and hence the result.

\item[(iii)] By part (iii) of Theorem \ref{e2} we have $E^{pq}_1=0$ for all $q>0$. Hence the spectral sequence degenerates and in particular, the edge homomorphism is  an isomorphism.
\end{itemize}
\end{proof}

{\bf Example}. Let $\ell$ be a  prime number and $G=C_\ell$ be a cyclic group of order $\ell$. Then
$$H^i_{\lambda}(C_\ell,M)=\begin{cases} H^i(C_\ell,M), \ {\rm if} \ i\leq \ell-1, \\ 0, \ {\rm if } \ {i\geq \ell}.\end{cases}$$
In fact, the case when $i\geq \ell$ follows from Remark \ref{26}, while the case $i\leq \ell-1$ follows from part (iii) of Theorem \ref{e2}.

\section*{Acknowledgements}
The paper was written during the author's postdoctoral fellowship at the University of Southampton. The author would like to thank the staff of the School of Mathematics, especiallly Prof. J. Brodzki, for providing me with excellent conditions to work.

\end{document}